\documentclass[oneside,a4paper,11pt]{article}
\usepackage{comment}
\usepackage{mathtools}
\usepackage{amsmath}
\allowdisplaybreaks[4]
\usepackage{csquotes}
\usepackage{amsthm}
\usepackage{xcolor}
\usepackage{amssymb}
\usepackage{mathrsfs}
\usepackage{color}
\usepackage{bm}
\usepackage{fancyhdr}
\usepackage{geometry}
\usepackage{hyperref}
\usepackage{enumerate}
\usepackage{graphicx}
\usepackage{subcaption}
\usepackage{float}

\theoremstyle{definition}
\newtheorem{thm}{Theorem}[section]
\newtheorem{cor}[thm]{Corollary}
\newtheorem{lem}[thm]{Lemma}

\newtheorem{defn}[thm]{Definition}
\newtheorem{rem}[thm]{Remark}
\newtheorem{conj}[thm]{Conjecture}
\newtheorem{problem}[thm]{Problem}

\numberwithin{equation}{section}

\newcommand{\pp}[2]{\frac{\partial#1}{\partial#2}}

\newcommand{\vol}{\mathrm{Vol}}
\newcommand{\CC}{\mathbb{C}}
\title{The rigidity of Doyle circle packings on the infinite hexagonal triangulation}
\author{
        Bobo Hua~~
	Puchun Zhou
}

\date{}
\usepackage{fancyhdr}
\providecommand{\classification}[1]
{
	\small	
	\textbf{Mathematics Subject Classification (2020):} #1
}

\begin{document}
	\maketitle
	\begin{abstract}
 
Peter Doyle conjectured that locally univalent circle packings on the hexagonal lattice only consist of regular hexagonal packings and Doyle spirals, which is called the Doyle conjecture. In this paper, we prove a rigidity theorem for Doyle spirals in the class of infinite circle packings on the hexagonal lattice whose radii ratios of adjacent circles have a uniform lower bound in one fixed lattice direction. This gives a partial answer to the Doyle conjecture. Based on a new observation that the logarithm of the radii ratio of adjacent circles is a weighted discrete harmonic function, we prove the result via the Liouville theorem for discrete harmonic functions. 

Moreover, we establish a Ring Lemma for locally univalent circle packings with bounded degree, assuming that their carrier surfaces are the universal covers of $\CC-\{0\}$, and the visual angles of all circles viewed from the origin are uniformly bounded away from $\pi$. This provides a new geometric sufficient condition for the rigidity of Doyle spirals.
\\
		\classification{52C25,52C26,57M50
		}	
		
	\end{abstract}
\section{Introduction}\label{sec:1}

A circle packing is a configuration consisting of tangent circles, which Thurston rediscovered in \cite{thurston1976geometry} for approximating Riemann maps with discrete patterns \cite{Rodin_Sullivan,MR1395721,MR1638772}. Circle packings build up a bridge between combinatorics and geometry. It is a powerful tool in graph theory, discrete geometry, topology, and probability theory,  see \cite{MR1207210,MR1419007,spielman1996disk, MR1438463,lovasz1999geometric,MR1781842,MR3010812,MR3970273}.
The definition of the circle packing is stated as follows.
\begin{defn}\label{circle_packing}
Let $\mathcal{T}=(V, E, F)$ be a simplicial $2$-complex with the sets of vertices $V$, edges $E$ and faces $F$. Let $\Sigma$ be an oriented surface with constant Gaussian curvature. Let $\mathcal{P}=\{C_v\}_{v\in V}$ be a configuration consisting of circles $\{C_v\}_{v\in V}$ in $\Sigma,$ where $C_v$ is associated to the vertex $v\in V$. The configuration $\mathcal{P}$ is said to be a circle packing of $\mathcal{T}$ if the following statements hold.
\begin{enumerate}
    \item Two circles $C_v$ and $C_w$ are externally tangent in $\Sigma$ to each other whenever $\{v,w\}$ is an edge of $\mathcal{T}$.
    \item Three circles $C_{v_1},C_{v_2},C_{v_3}$ form a positively oriented triple in $\Sigma$ whenever $(v_1,v_2,v_3)$ forms a positively oriented face of $\mathcal{T}.$
\end{enumerate}
By $v\sim w$ we denote $\{v,w\}\in E$.
\end{defn}
 Let $\mathcal{P}$ be a circle packing of a surface $\Sigma$ with constant curvature. For an interior vertex $v$, we denote by $N(v)=\{w\in V:w\sim v\}$ neighbors of the vertex $v.$ A \textbf{flower} $\mathcal{F}(v)$ of a vertex $v$ in $\mathcal{P}$ refers to the set of circles consisting of $C_v$ and circles associated to neighbors of $v.$ For each oriented triple $(v_1,v_2,v_3)$, we denote by $T_{v_1v_2v_3}$ the geodesic triangle obtained by connecting centers of circles $C_{v_1},~C_{v_2},~C_{v_3}$ in $\Sigma$.
 
The carrier of $\mathcal{P}$ is defined by 
$$\mathrm{Carr}_\Sigma(\mathcal{P})=\cup_{(v_1,v_2,v_3)\in F}T_{v_1v_2v_3}.$$
A circle packing $\mathcal{P}$ is called \textbf{locally univalent} if for every interior
 vertex $v$, the faces in $\mathrm{Carr}_\Sigma(\mathcal{P})$ corresponding to the flower $\mathcal{F}(v)$ have mutually disjoint
 interiors. In other words, the carrier of $\mathcal{F}(v)$ is embedded into the surface $\Sigma$.
  If all circles in $\mathcal{P}$ have disjoint interiors, we call $\mathcal{P}$ a \textbf{univalent circle packing}. See Stephenson \cite{Stephenson_intro} for definitions. We call $\mathcal{P}$ \textbf{a circle packing with univalent flowers}, if for each vertex $v\in V$, the flower $\mathcal{F}(v)$ of $\mathcal{P}$ is a univalent circle packing. One can see the difference between univalent and non-univalent flowers in Figure \ref{univalent} and Figure \ref{non-univalent}. Note that the set of univalent circle packings is a proper subset of circle packings with univalent flowers, and the latter is a proper subset of locally univalent circle packings.
 The univalence and local univalence stand for embedding and immersion of circle packings on surfaces.
\begin{figure}
     \centering
     \begin{subfigure}[c]
     {0.38\textwidth}
     \captionsetup{justification=centering}
         \centering
         \includegraphics[width=1.00\textwidth]{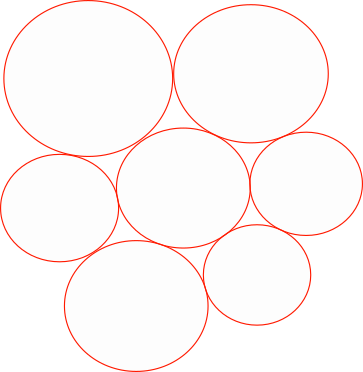}
         \caption{\small  A univalent flower.}
         \label{univalent}
     \end{subfigure}
     \begin{subfigure}[c]{0.53\textwidth}
\captionsetup{justification=centering}
         \centering
\includegraphics[width=0.65\textwidth]{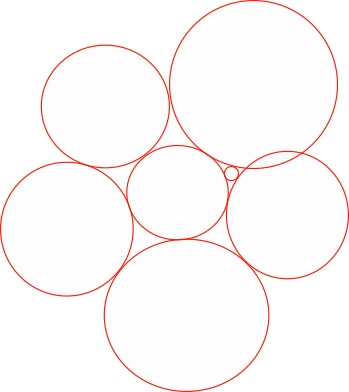}
        \caption{\small A non-univalent flower. }
         \label{non-univalent}
     \end{subfigure}
         \caption{\small Univalent and non-univalent flowers.}
        \label{abcde}
\end{figure}
In the rest of the paper, the local univalence is always assumed.

Now we consider the existence and uniqueness results of circle packings. For a finite triangulation $\mathcal{T}$ on a compact surface $\Sigma$, 
the first result of the existence and uniqueness was proven by Koebe for circle packings on the sphere in \cite{koebe1936origin}, and later the theorem was proven by Andreev and Thurston independently, see \cite{andreev1970convex1, thurston1976geometry}, which is now called the Koebe-Andreev-Thurston theorem. After that, Beardon and Stephenson extended the theorem and proved a uniformization theorem for univalent circle packings \cite{Beardon_stephenson}.
The rigidity of circle packings on triangulations with infinitely many triangles is an important topic. When $\mathcal{T}_H$ is a hexagonal triangulation of the complex plane $\mathbb{C}$, Rodin and Sullivan \cite{Rodin_Sullivan} proved that univalent circle packings of $\mathcal{T}_H$ consist of \textbf{regular hexagonal packings}, i.e. radii of circles are the same, as shown in Figure \ref{regular}. 
 \begin{figure}[htbp]
\centering
\includegraphics[scale=0.40]{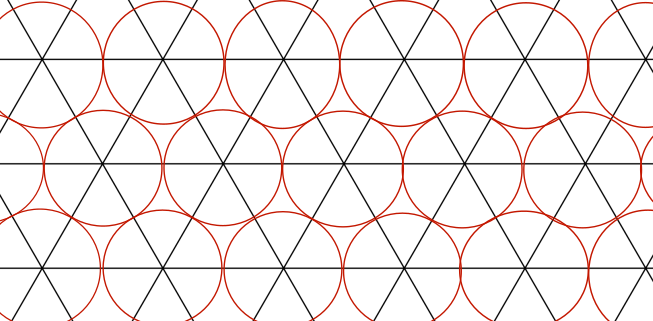}
\captionof{figure}{\small A part of a regular hexagonal packing.}
  \label{regular}
\end{figure} 
Schramm \cite{Schramm_oded} proved that for an infinite planar triangulation $\mathcal{T}$ and a univalent circle packing $\mathcal{P}$ of $\mathcal{T}$ on the sphere $\mathbb{S}^2$ (or on the unit disk $\mathbb{U}$), if $\mathbb{S}^2\backslash \mathrm{Carr}_{\mathbb{S}^2}(\mathcal{P})$ ($\mathbb{U}\backslash \mathrm{Carr}_{\mathbb{U}}(\mathcal{P})$) is at most countable, then $\mathcal{P}$ is unique up to M\"obius transformations. To prove the result, Schramm used delicate topological arguments on the plane. Later, He \cite{HE} proved the rigidity of locally finite disk patterns with the help of discrete extremal lengths of graphs, which is stated as follows in the case of circle packings. A uniformization theorem of univalent circle packings of infinite planar triangulations was proved by He and Schramm, see \cite{He_schramm}. 
\begin{thm}[He \cite{HE}]\label{he}
    Let $\mathcal{T}$ be an infinite triangulation of the open disk. Assume that $\mathcal{P}$ and $\mathcal{P}^*$ are univalent circle packings of $\mathcal{T}$ in $\mathbb{C}$, and $\mathrm{Carr}_{\CC}(\mathcal{P})=\mathbb{C}$. Then there exists a Euclidean similarity $f:\mathbb{C}\rightarrow\mathbb{C}$ such that $\mathcal{P}^*=f(\mathcal{P}).$
\end{thm}
For other rigidity results, we refer readers to \cite{MR1489142,MR3457760}.

Now we turn to locally univalent circle packings.
Uniqueness results proved by He and Schramm, however, do not apply in this general setting.
Hence it is important to consider moduli spaces of locally univalent circle packings with given combinatorics, e.g. the hexagonal triangulation $\mathcal{T}_H$.
\begin{defn}
Let $H$ be a lattice on the plane $\mathbb{C}$ given by
$$H=\{v_{m,n}=m+ne^{\frac{\pi i}{3}}:m,n\in \mathbb{Z}\}.$$
The hexagonal triangulation (or lattice) $\mathcal{T}_H$ is obtained from connecting vertex pairs $\{v_{m,n},v_{m',n'}\}$ whenever $|v_{m,n}-v_{m',n'}|=1$, $\forall m,n,m',n'\in\mathbb{Z}$. We denote by $H$, $E_H$, and $F_H$ the sets of vertices, edges, and faces of $\mathcal{T}_H.$
\end{defn}
Aside from regular hexagonal packings in Figure \ref{regular}, there is another class of locally univalent circle packings on $\mathcal{T}_H$ called Doyle spirals. Doyle spirals were used to model phyllotaxis by Gerrit van Iterson in 1907 \cite{jean1983introductory}, as shown in Figure \ref{Spiral}. They were named after Doyle for his important contribution to their mathematical construction. See also \cite{doyl_study1,doyl_study2} for earlier mathematical considerations of Doyle spirals.
 \begin{figure}[htbp]
\centering
\includegraphics[scale=0.40]{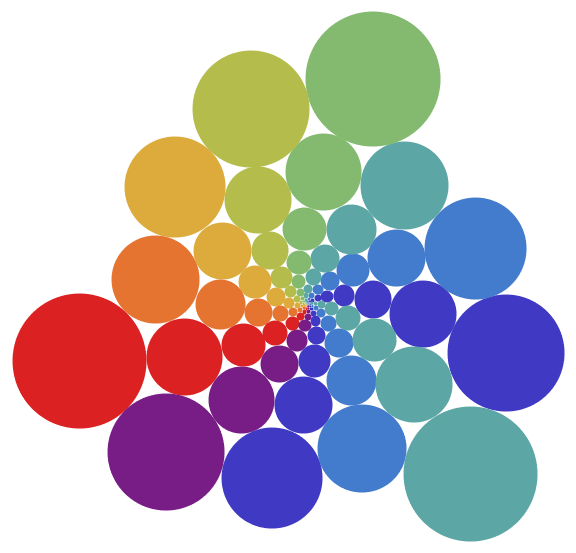}
\captionof{figure}{\small An example of Doyle spiral.}
  \label{Spiral}
\end{figure} Let $r_{m,n}$ be the radius of the circle at vertex $v_{m,n}$. Doyle spirals are defined as below.
\begin{defn}\label{doyle_spiral_defn}
   For any $r_0>0$ and $x,y\in\mathbb{R}_+$, there exists a circle packing of the hexagonal triangulation $\mathcal{T}_H$, denoted by $\mathcal{P}_{r_0,x,y}$,   with radii given by
\begin{align}\label{spiral_radii}
   r_{m,n}=r_0x^my^n,~\forall m,n\in \mathbb{Z}.
\end{align}
These circle packings determined by three parameters $r_0,~x$ and $y$ are called \textbf{Doyle spirals}.
\end{defn}
Since no other locally univalent circle packings of $\mathcal{T}_H$ can be found in numerical experiments, the following conjecture was proposed by Peter Doyle \cite[Open~Question~14.9]{Stephenson_intro}, which is one of the earliest questions in circle packing according to \cite[p.397]{MR2963949}, see also \cite{circlepattern_square,bobenko1}.
\begin{conj}\label{conj}
    There are no locally univalent circle packings of the hexagonal triangulation $\mathcal{T}_H$ other than regular hexagonal packings and Doyle spirals. 
\end{conj}
In the case of \textbf{coherent circle packings}, where any two circles are identical or have mutually disjoint interiors in $\mathbb{C}$, Beardon, Dubejko and Stephenson \cite{Stephenson_D} gave an affirmative answer to the Doyle conjecture.
However, in the general cases, the conjecture remains open. As a related result, Callahan and Rodin proved that locally univalent hexagonal packings of $\mathcal{T}_H$ form regularly exhaustible surfaces in \cite{Rodin93}, to which Ahlfors's value distribution theory applies. See also \cite{Callahan1993} for Callahan's Ph.D. thesis, which contains a more detailed discussion. 

Aside from circle packings, circle patterns on the hexagonal lattice or the square grid also admit structures like Doyle spirals, see \cite{circlepattern_square, bobenko1,Bobenko2}. Those
patterns are closely related to the theory of discrete analytic functions, where Doyle spirals can be viewed as discrete analogs of the exponential function; see e.g. \cite{Stephenson_D}.
In particular, Schramm \cite{circlepattern_square} proved that the Doyle conjecture is false for orthogonal patterns on the square grid by finding a special family of circle patterns called erf-like patterns. 
In this paper, we prove that Conjecture \ref{conj} is true for a large class of locally univalent hexagonal circle packings as follows.
\begin{thm}\label{thm:Doyle_bounded_ratio} 
    Let $\mathcal{P}$ be a locally univalent circle packing of hexagonal triangulation $\mathcal{T}_H$ on the complex plane $\mathbb{C}$. If there exists a positive constant $c$ such that
    \begin{align}\label{hanack}
\frac{r_{m+1,n}}{r_{m,n}}\ge c,~\forall m,n\in \mathbb{Z},
    \end{align}
    then it is either a regular hexagonal packing or a Doyle spiral.
\end{thm}
 Note that the condition \eqref{hanack} is guaranteed by the Ring Lemma, Lemma~\ref{ring}, when all the flowers are univalent. Therefore, we have the following corollary.
\begin{cor}\label{univalent_flower}
    Any circle packing with univalent flowers of the hexagonal triangulation $\mathcal{T}_H$ on the complex plane $\mathbb{C}$ is either a regular hexagonal packing or a Doyle spiral.
\end{cor}

\begin{rem}
\begin{enumerate}[a.]
   \item Rigidity results of circle packings of $\mathcal{T}_H$ proved by Rodin and Sullivan \cite{Rodin_Sullivan}, and by Beardon, Dubejko and Stephenson \cite{Stephenson_D} are direct corollaries of the above result. The class of circle packings with univalent flowers is larger than that of univalent circle packings. In the former case, interiors of two different circles may have a non-empty intersection if they are not contained in a flower.
\item Since Theorem \ref{thm:Doyle_bounded_ratio} is about the rigidity of locally univalent circle packings, it is an extension of previous rigidity results \cite{Schramm_oded,He_schramm,HE} for the hexagonal lattice case.
\item 
In Theorem \ref{thm:Doyle_bounded_ratio}, we only assume the lower bound of radii ratios in one direction, having no assumptions in other directions, e.g. $\frac{r_{m,n+1}}{r_{m,n}}$ and $\frac{r_{m+1,n-1}}{r_{m,n}}$.
\item
   Schramm proved a similar result for $\mathrm{SG}$-patterns, which are orthogonal circle patterns on square grids, see \cite[Theorem 7.2 and Corollary 7.3]{circlepattern_square}. Taking advantage of the special structure of square grids, he proved that the collection of entire $\mathrm{SG}$-patterns in the sphere is infinite-dimensional, and a certain class of $\mathrm{SG}$-patterns are SG-Doyle spirals.
   \end{enumerate}
\end{rem}

Corresponding to the work of Rodin and Callahan, Schramm further posed the following problem as a ``dessert''.

\begin{problem}[\cite{circlepattern_square}]\label{problem_schramm}
	Suppose that the surface determined by an entire $SG$ immersion $\mathcal{C}$ is the universal cover of $\CC-\{0\}$. Is $\mathcal{C}$ necessarily an SG-Doyle spiral?
\end{problem}
For a locally univalent hexagonal packing of $\mathcal{T}$ in $\CC$, each face $(v_i,v_j,v_k)$ in $\mathcal{T}$ can be realized as a Euclidean triangle with edge lengths $r_{v_i}+r_{v_j}$, $r_{v_j}+r_{v_k}$ and $r_{v_i}+r_{v_k}$. Gluing these triangles along corresponding edges, we obtain a non-compact surface $S$ with Euclidean metric. We call $S$ the carrier surface of $\mathcal{P}$. On $S$, there is also a circle packing $\tilde{\mathcal{P}}=\{\tilde{C}_v\}_{v\in V}$ whose radii are the same as $\mathcal{P}$. And there is a natural projection
\[
\pi:S\to \mathrm{Carr}_{\CC}(\mathcal{P})\subset\CC,
\]
such that $\pi(\tilde{C}_v)=C_v$. In the case of hexagonal circle packing, if we view the carrier surface $S_0$ of the regular hexagonal packing as the discrete complex plane, which is identified with $\CC$, then for every hexagonal circle packing $\mathcal{P}$ on $\CC$ whose carrier surface is $S$, there is a piecewise linear map 
\[
g_\mathcal{P}:S_0\to S,
\]
which maps each equilateral triangle of $S_0$ affinely to the corresponding Euclidean triangle of $S$. Therefore, a problem similar to Problem \ref{problem_schramm} can be stated as follows. 

\begin{problem}\label{problem_hex}
	Let $\mathcal{P}$ be a locally univalent circle packing of $\mathcal{T}_H$ whose carrier surface is $S$. If $\pi:S\rightarrow \CC-\{0\}$ is a universal cover of $\CC-\{0\}$, is $\mathcal{P}$ necessarily a Doyle spiral?
\end{problem}
\begin{rem}
	The map $\pi$ defined here is also the developing map of the carrier surface $S$.
\end{rem}

This covering condition yields two logarithmic singularities of the inverse projection, lying over $0$ and $\infty$. For the arguments below, we use the following notion of a finite singular value of a circle packing.

\begin{defn}[Singular values]
	Let $\mathcal{P}$ be a locally univalent circle packing with
	carrier surface $S$ and projection $\pi:S\to\mathbb{C}$.
	We call $a\in\mathbb{C}$ a singular value of $\mathcal{P}$ if,
	for every $\varepsilon>0$, there exists a connected component
	of $\pi^{-1}(B_{\CC}(a,\varepsilon))$ that is not relatively compact
	in $S$. We denote the set of these values by
	$\operatorname{SV}(\mathcal{P})$.
\end{defn}

\begin{rem}
	Here we use the term ``singular values'' to include limit
	points of ``asymptotic values'', which is defined in  \cite[Section~1]{Eremenko2021Inverse}.
\end{rem}

In view of Theorem~\ref{thm:Doyle_bounded_ratio}, a natural problem arises.
\begin{problem}
	Are there simple sufficient conditions on the value distribution of circle packings that imply a Ring Lemma, or, more specifically, a gradient estimate for the logarithmic radii?
\end{problem}

In Section \ref{sec:Ring_lemma_with_special_carrier}, we further prove the following Ring Lemma for circle packings with bounded degree whose carrier omits $0$.
\begin{lem}\label{lem:extended_ring}
	Let $\mathcal{T}=(V,E,F)$ be a locally finite infinite triangulation of $\CC$ with vertex degree bounded by $N$. Assume that $\mathcal{P}$ is a locally univalent circle packing of $\mathcal{T}$ whose carrier does not contain the origin $0$. If $SV(\mathcal{P})=\{0\}$, and 
	\begin{align}\label{equ:bounded_log}
		\tau(\mathcal{P}):=\sup_{v\in V}\frac{r_v}{|c_v|-r_v}<\infty,
	\end{align}
	where $c_v$ is the center of the circle $C_v$. Then 
	\begin{align}
		C^{-1}\le \frac{r_v}{r_w}\le C,\qquad \forall v\sim w,
	\end{align}
	where $C=C(N,\tau(\mathcal{P}))$ is a constant.
\end{lem}
\begin{rem}
	This lemma shows that a pointwise estimate implies the gradient estimate when a value distributional property is assumed.
\end{rem}

The proof of this special Ring lemma uses the Gromov-Hausdorff convergence of doubly rooted graphs and blow-down technique. This method goes back to He's proof of Ring lemma \cite{HE}, and was generalized in one recent article \cite{GeYuZhou2025}

Combining Theorem \ref{thm:Doyle_bounded_ratio} with Lemma \ref{lem:extended_ring} we obtain the following result.

\begin{thm}
	Let $\mathcal{P}$ be a locally univalent circle packing of $\mathcal{T}_H$ whose carrier surface is $S$. If $\pi:S\rightarrow\CC-\{0\}$ is the universal covering map, and
	  \begin{align}\label{equ:bounded_log1}
		\tau(\mathcal{P}):=\sup_{v\in V}\frac{r_v}{|c_v|-r_v}<\infty,
	\end{align}
	then $\mathcal{P}$ is necessarily a Doyle spiral.
	Moreover, condition \eqref{equ:bounded_log1} is equivalent
	to the visual angles of all circles from the origin
	being uniformly bounded away from $\pi$.
\end{thm}
\begin{rem}
	The quantity $\tau(\mathcal{P})$ was originally defined in He's paper \cite[(6.5)-(6.6)]{HE}, see the proof of Lemma 6.3, and determines the angle between the two tangent rays drawn from the origin to each circle. 
\end{rem}
\begin{rem}
		Xin Nie has also obtained an alternative analytic proof of the above theorem, which uses a totally different method without proving the new Ring lemma; see \cite{Nie2026}.
\end{rem}

To prove the rigidity of infinite circle packings, value distributional assumptions such as the univalence and the countability for the complement of the carriers in $\mathbb{S}^2$ are crucial for the arguments in \cite{Schramm_oded,He_schramm,HE}. For the Doyle conjecture, one only has the assumption of locally univalent structure, analogous to the immersion in geometry, for which there is no global information for the carrier. The key idea in this paper is to derive the equation for radii ratios directly from the nonlinear equation of the circle packing \eqref{equation}, a discrete analog of an elliptic partial differential equation, and to apply the Liouville theorem for deriving the rigidity result. The proof strategy is as follows: we first define two spatial difference operators $D_1$ and $D_2$ on functions of the hexagonal lattice $H$. Let $u=(u_{m,n})_{m,n\in\mathbb{Z}}$ be the logarithms of the radii $r=(r_{m,n})_{m,n\in\mathbb{Z}}$ of circles in $\mathcal{P}$. As a key observation, we prove that $D_1 u$  is a discrete harmonic function for some uniformly bounded weights defined on $\mathcal{G}_H=(H,E_H)$, with the help of equations \eqref{equation}. Next, since $D_1u$ is bounded below by \eqref{hanack}, we prove that $D_1u\equiv k_1$, for some constant $k_1$ thanks to the following well-known Liouville theorem, a corollary of the Nash-Williams test for the recurrence of the associated weighted random walk.
\begin{lem}\label{harmonic_lemma}
For a hexagonal graph with uniformly bounded positive symmetric edge weights, any positive superharmonic function is constant.
\end{lem}
 Finally, one can prove that $D_2u$ is also a constant by the circle packing structure. Hence $u$ is an affine function, which proves that the circle packing $\mathcal{P}$ is a Doyle spiral. 

In the proof of Theorem~\ref{he}, given two circle packings $\mathcal{P}$ and $\mathcal{P}^*$, He \cite{HE} constructed a continuous family $(\mathcal{P}(t))_{t\ge 0}$ connecting them, and derived a discrete harmonic function by differentiating with respect to $t$. Different from that, our construction is taking the spatial difference of one solution, which is motivated by gradient estimates for nonlinear partial differential equation. Note that a similar consideration appears in the proofs of rigidity results of infinite hexagonal triangulations in the discrete conformal geometry theory, \cite{wu2015rigidity,MR4389487,convergence_discrete_conformal}, where conjectures similar to the Doyle conjecture are still open.

The paper is organized as follows: in the next section, we recall basic facts of circle packings and discrete harmonic functions. Section~\ref{sec:var} is devoted to the variational principle of circle packings and the proof of the main rigidity theorem \ref{thm:Doyle_bounded_ratio}. Finally, in Section \ref{sec:Ring_lemma_with_special_carrier}, we will prove Lemma \ref{lem:extended_ring}.

\section{Preliminaries}\label{sec:2}
For a connected graph $G=(V,E)$, we denote by
$d_G(v,w)$ the combinatorial distance between
two vertices $v,w\in V$, namely, the minimum
number of edges in a path joining them.
For $R\geq 0$, we define the closed combinatorial ball by
\[
B_R^G(v):=\{w\in V:d_G(v,w)\leq R\}.
\]
On the other hand, for $z\in\mathbb C$ and $r>0$,
we denote by
\[
B_{\mathbb C}(z,r)
:=\{w\in\mathbb C:|w-z|<r\}
\]
the open Euclidean disk, and by
$\overline{B_{\mathbb C}(z,r)}$
the corresponding closed disk.
\subsection{Locally univalent circle packings of triangulations and the Ring Lemma}
Let $\mathcal{T}=(V,E,F)$ be a triangulation. We assume that $\mathcal{P}$ is a circle packing of $\mathcal{T}$. Let
$r\in\mathbb{R}_+^V$ be the radius vector.
Then $r$ determines an intrinsic metric on triangulation. For each face $T$ of $\mathcal{T}$ with vertices $v_1,v_2$ and $v_3$, there exists a unique Euclidean triangle with edge lengths $r_{v_1}+r_{v_2},r_{v_1}+r_{v_3}$ and $r_{v_2}+r_{v_3}$ as shown in Figure \ref{ccp}.
We denote by $\theta_{v_i}^T$ the inner angle of $T$ at 
the vertex $v_i$. Then $\theta_{v_i}^T$ is a function of $r$ which depends only on $r_{v_1},r_{v_2},r_{v_3}.$ We assume that $\mathcal{P}$ is a locally univalent circle packing on the plane. By the definition of the local univalence, one gets the system of equations for a circle packing,
\begin{align}
\sum_{T:T>v}\theta_v^T(r)=2\pi, ~\forall v\in V, \label{equation}
\end{align}
where $T>v$ means $v$ is a vertex of $T$.

One easily verifies by \eqref{equation} that a Doyle spiral defined in Definition \ref{doyle_spiral_defn} is a locally univalent circle packing. An important feature of circles in a univalent flower is that they satisfy the Ring Lemma of Rodin and Sullivan \cite{Rodin_Sullivan}.
\begin{lem}[\cite{Rodin_Sullivan} Ring Lemma]\label{ring}
    Let $\mathcal{P}$ be a circle packing of a triangulation $\mathcal{T}$. Set $v$ be an interior vertex and $w\in N(v)$. Then if $|N(v)|= k$ for some $k\ge 3$, and the flower $\mathcal{F}(v)$ is univalent, then there exists a positive constant $\mathcal{R}(k)$ that only depends on $k$ such that $$\mathcal{R}(k)\leq \frac{r_w}{r_v}.$$  Moreover, $$\mathcal{R}(k+1)<\mathcal{R}(k)\leq 1.$$
\end{lem}
    There are various proofs of Ring Lemma in the literature, see e.g. \cite[Appendix B]{Stephenson_intro}.
We denote by $d(v)=|N(v)|$ the number of neighbors of a vertex $v$, called the \textbf{degree} of the vertex $v$. With the help of the Ring Lemma, we obtain the following corollary.
    \begin{cor}\label{bounded degree}
        Let $\mathcal{P}$ be a circle packing with univalent flowers of a triangulation $\mathcal{T}=(V,E,F)$. If $K:=\sup_{v\in V} d(v)<\infty$, then for inner vertices, we have
        \[
        \mathcal{R}(K)\le\frac{r_v}{r_w}\le\frac{1}{\mathcal{R}(K)}, ~~\forall v\sim w.
        \]
    \end{cor}

With the help of Corollary \ref{bounded degree}, we deduce from our main result, Theorem \ref{thm:Doyle_bounded_ratio}, the following rigidity results in Rodin, Sullivan \cite{Rodin_Sullivan} and Beardon, Dubejko, Stephenson \cite{Stephenson_D}. 
\begin{thm}[\cite{Rodin_Sullivan}]
    If $\mathcal{P}$ is a univalent packing of $\mathcal{T}_H$ in $\mathbb{C},$ then $\mathcal{P}$ is a regular hexagonal packing.
\end{thm}
\begin{thm}[\cite{Stephenson_D}]
    If $\mathcal{P}$ is a coherent circle packing of $\mathcal{T}_H$ in $\mathbb{C}$, then $\mathcal{P}$ is a Doyle spiral.
\end{thm}
 The first result follows from the fact that the only univalent Doyle spirals are regular hexagonal packings. The second result is a consequence of Corollary \ref{univalent_flower} and that all flowers of a coherent circle packing are univalent.
 
    We point out that flowers of a Doyle spiral $\mathcal{P}_{r_0,x,y}$ are not always univalent, for all $x$ and $y$. Indeed, when $x$ is large enough, the flowers are not univalent by the Ring Lemma.

\subsection{Discrete harmonic functions on graphs}
For a graph $G=(V,E)$, two vertices $v,w$ are called neighbors, denoted by $v\sim w$ if $\{v,w\}\in E$.
\begin{defn} 
     Let $G=(V,E)$ be a graph,  $\eta:E\rightarrow\mathbb{R}_+$ be weights on edges, with $$\eta_{vw}=\eta_{wv},\quad\forall v\sim w.$$ A function $u$ on $V$ is called a weighted discrete harmonic (superharmonic resp.) function on the weighted graph $(G,\eta)$ if and only if 
    \[
    \sum_{w\sim v}\eta_{wv}(u_w-u_v)=0\ (\le0~\mathrm{resp.}),~~\forall v\in V.
    \]
    \end{defn}
Here we give a simple proof of Lemma \ref{harmonic_lemma}.

\begin{proof}
For any subset $W\subset V$, we write $$\mathrm{Vol}(W)=\sum_{v\in W}\sum_{w\sim v}\eta_{vw}.$$
Since weights are uniformly bounded, one easily verifies that $$\vol(B_n^G(v))\le Cn^2,\quad \forall n\geq 1.$$ i.e. the graph is of quadratic volume growth. By a well-known theorem, see e.g.  \cite[Theorem 6.13]{grigor2018introduction}, the associated weighted  random walk is recurrent. This yields that any positive superharmonic function is constant, see \cite[Theorem 1.16]{woess2000random}.
\end{proof}

\section{Proof of Theorem \ref{thm:Doyle_bounded_ratio}}\label{sec:var}
First, we briefly recall the variational principle for Euclidean triangles with circle packing metrics.
Let $T$ be a triangle with vertices $v_1,v_2$ and $v_3$ obtained by three tangent circles $C_1,C_2$ and $C_3$ with radii $r_1,r_2$ and $r_3$, as shown in Figure \ref{ccp}. 
 \begin{figure}[htbp]
\centering
\includegraphics[scale=0.60]{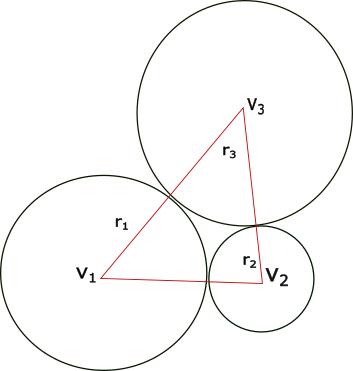}
\captionof{figure}{\small A triangle with the circle packing metric.}
  \label{ccp}
\end{figure} 
 We denote by $l_{12},l_{13}$ and $l_{23}$ the side lengths of $T$, which equal $r_1+r_2,r_1+r_3$ and $r_2+r_3$, respectively. We denote by $\theta_1,\theta_2$ and $\theta_3$ the inner angles at $v_1,v_2$ and $v_3$ of $T$.
Colin de Verdi\`ere \cite{MR1106755} used the following well-known lemma to prove the Koebe-Andreev-Thurston Theorem, via the variational principle.
\begin{lem}\label{vari}
    Set $u_i=\ln r_i$ for $i=1,2,3.$  Then $\theta_i(i=1,2,3)$ is a smooth function of $u_1,u_2$ and $u_3$. Moreover, the following statements hold.
    \begin{enumerate}[(a)]
        \item $\pp{\theta_i}{u_j}>0$,~  $\forall i\neq j,$ and $\pp{\theta_i}{u_i}<0,~\forall i=1,2,3.$
        \item 
        $
        \pp{\theta_i}{u_j}=\pp{\theta_j}{u_i}~\forall i\neq j.$
        \item 
        $\pp{\theta_i}{u_1}+\pp{\theta_i}{u_2}+\pp{\theta_i}{u_3}=0,~\forall i=1,2,3.$
    \end{enumerate}
\end{lem}
In Lemma \ref{vari}, the assertion (c) is obtained by the fact that for any $i$
\begin{align}\label{affine_prop}
\theta_i(u_1,u_2,u_3)=\theta_i(u_1+\lambda,u_2+\lambda,u_3+\lambda),~\forall \lambda\in \mathbb{R},
\end{align}
since angles of a triangle are invariant under the scaling of side lengths. By \eqref{affine_prop} and the cosine law of Euclidean triangles, the angle $\theta_i$ can be rewritten as
\begin{align}\label{normal_form}
\theta_i(u_1,u_2,u_3)=\Theta(u_j-u_i,u_k-u_i), ~~\text{with} ~\{i,j,k\}=\{1,2,3\},
\end{align}
where \begin{align}\label{change_function}
\Theta(x_1,x_2)=\arccos\frac{(1+e^{x_1})^2+(1+e^{x_2})^2-(e^{x_1}+e^{x_2})^2}{2(1+e^{x_1})(1+e^{x_2})}.
\end{align}
By \eqref{normal_form} and \eqref{change_function}, we get the following useful estimate, which is crucial for the proof of Theorem \ref{thm:Doyle_bounded_ratio}. 
\begin{lem}\label{uniform}
Set $u_i=\ln r_i$ for $i=1,2,3.$ Let $\theta_i=\theta_i(u_1,u_2,u_3)(i=1,2,3)$ be inner angles at the vertex $v_i$ of the triangle. Then for $i\neq j$ 
\[
    \pp{\theta_i}{u_j}<1.
\]
\end{lem}
\begin{proof}
  By direct calculation, we have 
    $$
    \pp{\Theta}{x_1}=\frac{1}{1+e^{x_1}}\sqrt{\frac{e^{x_1+x_2}}{1+e^{x_1}+e^{x_2}}}< 1.
    $$
\end{proof}
   A similar estimate can be found in \cite{HE} for proving the rigidity of univalent circle packings on the plane. Indeed, He proved that $\pp{\theta_i}{u_j}\le K\theta_i$ for some uniform constant $K.$  We also remark that such an estimate is useful for defining combinatorial Calabi flows, see \cite{ge_phd}.

Let $\mathcal{T}_H=(H,E_H,F_H)$ be the hexagonal triangulation.
Before proving Theorem \ref{thm:Doyle_bounded_ratio}, we introduce some basic operators on $\mathbb{R}^H,$ the space of functions on $H$.
\begin{defn}
   The difference operators $D_1$ and $D_2$ are defined for any function $f=(f_{m,n})_{m,n\in\mathbb{Z}}\in \mathbb{R}^H,$
    \begin{align*}
    D_1f_{m,n}=f_{m+1,n}-f_{m,n},\\
    D_2f_{m,n}=f_{m,n+1}-f_{m,n}.
    \end{align*}
\end{defn}
Let $\mathcal{P}$ be a locally univalent circle packing of the hexagonal triangulation $\mathcal{T}_H$ in $\mathbb{C}$. Set $r=(r_{m,n})_{m,n\in\mathbb{Z}}$ the radii of circles in $\mathcal{P}$. Let $u=(u_{m,n})_{m,n\in\mathbb{Z}}$ given by
\[u_{m,n}=\ln r_{m,n},\]
which are also called \textbf{discrete conformal factors} of the circle packing $\mathcal{P}$ in \cite{gu2008computational}. Then the conclusion of
Theorem \ref{thm:Doyle_bounded_ratio} is equivalent to the following result.
\begin{thm}\label{main}
        Let $\mathcal{P}$ be a locally univalent circle packing on the complex plane $\mathbb{C}$ of $\mathcal{T}_H$. Let $u=(u_{m,n})_{m,n\in\mathbb{Z}}$ be discrete conformal factors of $\mathcal{P}$. Assume that $$\inf_{m,n\in\mathbb{Z}} D_1u > -\infty,$$ then there exist two constants $k_1$ and $k_2$ such that
    \[
    D_iu\equiv k_i,~\forall i=1,2.
    \]
\end{thm}
First, let us define a translation operator $R$ on the triangulation $\mathcal{T}_H$. 

\begin{enumerate}
    \item 
    For any vertex $v=v_{m,n}\in H$,
    $R(v_{m,n}):=v_{m+1,n}.$
    \item For any edge $\{v_1,v_2\}$ and any face $\{v_1,v_2,v_3\}$ of $\mathcal{T}_H$, 
    \[
    R(\{v_1,v_2\}):=\{R(v_1),R(v_2)\},~R(\{v_1,v_2,v_3\}):=\{R(v_1),R(v_2),R(v_3)\}
    .\]
\end{enumerate}
Hence for any $f\in \mathbb{R}^H$,  $D_1f_v=f_{R(v)}-f_v.$

Now suppose that $\mathcal{P}$ is a locally univalent circle packing of the hexagonal triangulation $\mathcal{T}_H$. Let $u$ be discrete conformal factors of $\mathcal{P}$.
Fixing an arbitrary vertex $v\in H$, let $T=\{v_1,v_2,v_3\}$ be a triangle containing $v$ with $v_1=v,$ and let $R(T)$ be its translation as shown in Figure \ref{Translation}. The segment connecting $(u_{v_1},u_{v_2},u_{v_3})$ and $(u_{R(v_1)},u_{R(v_2)},u_{R(v_3)})$ in $\mathbb{R}^3$ is given by
\[
f_T(t)=(u_{v_1}+(u_{R(v_1)}-u_{v_1})t,u_{v_2}+(u_{R(v_2)}-u_{v_2})t,u_{v_3}+(u_{R(v_3)}-u_{v_3})t),~\forall t\in[0,1].
\]
 \begin{figure}[htbp]
\centering
\includegraphics[scale=0.35]{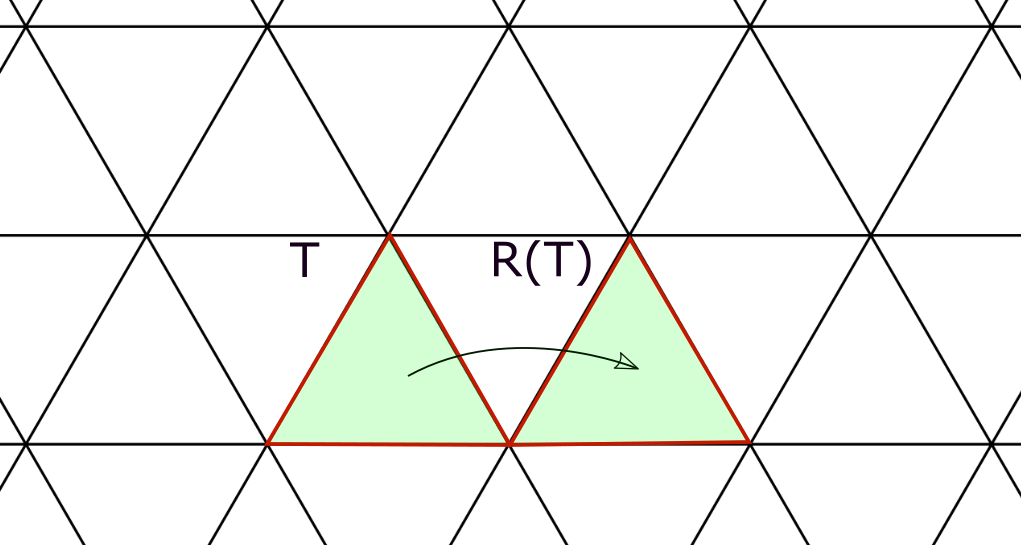}
\captionof{figure}{\small Triangles $T$ and $R(T)$.}
  \label{Translation}
\end{figure} 
Then by the Newton-Leibniz formula, for $\{i,j,k\}=\{1,2,3\}$, one has
\begin{align}\label{equ_difference}
&\theta_{R(v_i)}^{R(T)}-\theta_{v_i}^T=\int_0^1\nabla\theta_{v_i}(f_T(t))\cdot f'_T(t)\mathrm{d}t,\nonumber\\
&=(u_{R(v_i)}-u_{v_i})\int^1_0\pp{\theta_{v_i}(f_T(t))}{u_{v_i}}\mathrm{d}t
+(u_{R(v_j)}-u_{v_j})\int^1_0\pp{\theta_{v_i}(f_T(t))}{u_{v_j}}\mathrm{d}t\nonumber\\&+(u_{R(v_k)}-u_{v_k})\int^1_0\pp{\theta_{v_i}(f_T(t))}{u_{v_k}}\mathrm{d}t,\nonumber
\\&=D_1u_{v_i}\int^1_0\pp{\theta_{v_i}(f_T(t))}{u_{v_i}}\mathrm{d}t+D_1u_{v_j}\int^1_0\pp{\theta_{v_i}(f_T(t))}{u_{v_j}}\mathrm{d}t+D_1u_{v_k}\int^1_0\pp{\theta_{v_i}(f_T(t))}{u_{v_k}}\mathrm{d}t.
\end{align}
By the assertion (c) of Lemma \ref{vari}, the equation \eqref{equ_difference} can be written as \begin{align}\label{difference_normal}
    \theta_{R(v_i)}^{R(T)}-\theta_{v_i}^T=&\left(\int^1_0\pp{\theta_{v_i}(f_T(t))}{u_{v_j}}\mathrm{d}t\right)\left(D_1u_{v_j}-D_1u_{v_i}\right),\nonumber\\&+\left(\int^1_0\pp{\theta_{v_i}(f_T(t))}{u_{v_k}}\mathrm{d}t\right)(D_1u_{v_k}-D_1u_{v_i}).
\end{align}
For an edge $\{v_1,v_2\}\in E_{H},$ we denote by $F_{v_1,v_2}$ the set of faces that contain the edge $\{v_1,v_2\}.$ Note that $F_{v_1,v_2}$ consists of exactly two faces.
Since the equation \eqref{equation} holds at both $v$ and $R(v)$,
we obtain a linear equation of $D_1u$ at the vertex $v$ as follows
\begin{align}   0&=\sum_{T:T>R(v)}\theta_{R(v)}^T(u)-\sum_{T:T>v}\theta_v^T(u),\nonumber\\
&=\sum_{T:T>v}(\theta_{R(v)}^{R(T)}(u)-\theta_v^T(u)),\nonumber\\
&=\sum_{w:w\sim v}\eta_{vw}(D_1u_w-D_1u_v),\label{harmonic_deri1}
\end{align}
 where the last equation follows from \eqref{difference_normal}, and
\begin{align}\label{harmonic_weight}
    \eta_{vw}=\sum_{T\in F_{v,w}}\int^1_0\pp{\theta_{v}(f_T(t))}{u_w}\mathrm{d}t.
\end{align}
By the above observation, we have the key lemma.
\begin{lem}\label{bounded_weight}
Let $\mathcal{P}$ be a locally univalent circle packing on the complex plane $\mathbb{C}$ of $\mathcal{T}_H$. Let $u=(u_{m,n})_{m,n\in\mathbb{Z}}$ be discrete conformal factors of $\mathcal{P}$. Then $D_1u$ is a discrete harmonic function on the weighted graph $(\mathcal{G}_H,\eta),$ where $\mathcal{G}_H=(H,E_H)$ and
\[
\eta_{vw}=\sum_{T\in F_{v,w}}\int^1_0\pp{\theta_{v}(f_T(t))}{u_w}\mathrm{d}t,\forall v\sim w.
\]
Moreover, $\eta_{vw}=\eta_{wv}>0$ and
\[
\eta_{vw}<2,~\forall v\sim w.
\]
\end{lem}
\begin{proof}
    By equations \eqref{harmonic_deri1} and \eqref{harmonic_weight}, $D_1u$ is a harmonic function on the weighted graph $(\mathcal{G}_H,\eta)$. And the symmetry of the weight $\eta$ follows from the assertion (b) of Lemma \ref{vari}, which yields that for each $v\sim w$ and $T\in F_{v,w},$
    \[
    \pp{\theta_v(f_T(t))}{u_w}=\pp{\theta_w(f_T(t))}{u_v},~\forall t\in[0,1].
    \]
    Moreover, by the assertion (a) of Lemma \ref{vari}, $\eta_{vw}>0,~\forall v\sim w.$
   The last assertion follows from Lemma \ref{uniform}. 
\end{proof}
\begin{figure}[!ht]
\centering
\includegraphics[scale=0.25]{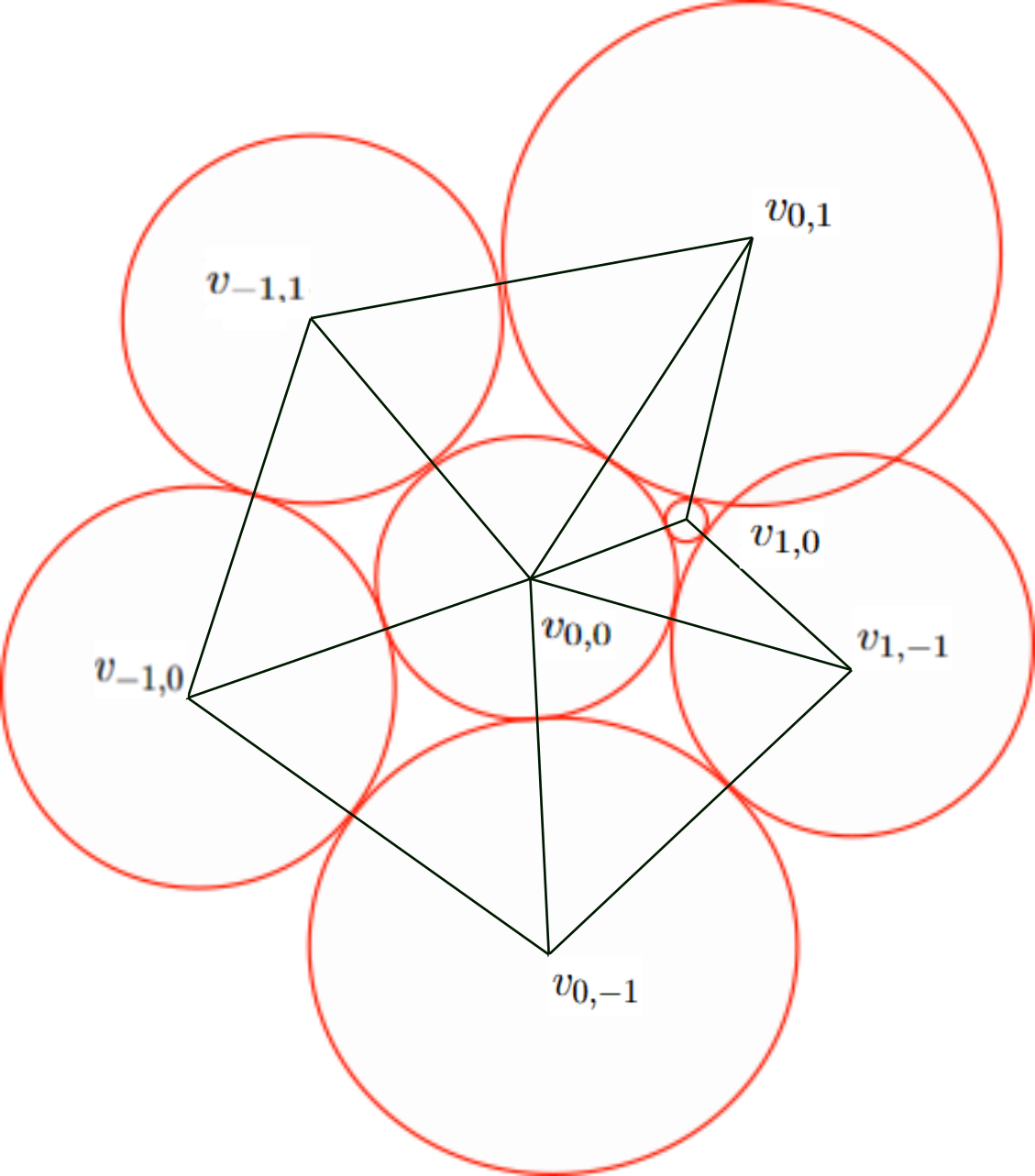}
\captionof{figure}{\small The flower $\mathcal{F}(v_{0,0})$.}
  \label{Flower}
\end{figure} 
Now we prove the main result.
\begin{proof}[Proof of Theorem \ref{main}]
By Lemma \ref{harmonic_lemma} and Lemma \ref{bounded_weight}, since $(\mathcal{G}_H,\eta)$ is a recurrent network and $D_1u$ is a discrete harmonic function that is bounded from below, $D_1u$ must be constant. Let $D_1u\equiv k_1$. Now we only need to prove that $D_2u$ is also constant. 

Without loss of generality, assume that $u_{0,0}=0.$ Then $u_{m,0}=k_1m.$ Write $a_1:=u_{0,1}$ and $a_2:=u_{0,-1}$. Then $$u_{m,1}=a_1+mk_1,~ u_{m,-1}=a_2+mk_1.$$
Let $C_{m,n}$ denote the circle associated with the vertex $v_{m,n}$.
Now consider the flower of $v_{0,0}$, which consists of circles $$C_{0,0},C_{1,0},C_{0,1},C_{-1,1},C_{-1,0},C_{0,-1},C_{1,-1},$$ as shown in Figure \ref{Flower}.
Then radii of 
 those circles can be represented by parameters $a_1,a_2$ and $k_1.$ We now prove that $a_1=-a_2.$ In the case $a_1=-a_2$, the flower $\mathcal{F}(v_{0,0})$ can be regarded as a flower in a Doyle spiral $\mathcal{P}_{1,e^{k_1},e^{a_1}}$. Therefore the equation \eqref{equation} holds at $v_{0,0}.$
 We denote by $F_1,F_2$ and $F_3$ the faces associated to triples $$(v_{-1,0},v_{0,0},v_{0,-1}),(v_{0,-1},v_{0,0},v_{1,-1}), \ \mathrm{and}\  
 (v_{1,-1},v_{0,0},v_{1,0}).$$ Then by Lemma \ref{vari}, $\theta_{v_{0,0}}^{F_1},\theta_{v_{0,0}}^{F_2}$ and $\theta_{v_{0,0}}^{F_3}$ are monotonically increasing with respect to $a_2.$ Therefore, the equation \eqref{equation} holds at $v_{0,0}$ if and only if $a_2=-a_1$. Thus, we prove that $D_2u_{0,0}=D_2u_{0,-1}=a_1.$ By induction, one easily sees that $D_2u\equiv a_1$. Therefore, the proof is completed.
\end{proof}

	\section{A Ring Lemma for locally univalent circle packing}
	\label{sec:Ring_lemma_with_special_carrier}
	
	Before proving Lemma \ref{lem:extended_ring}, we first introduce
	some basic notation.
	
	\begin{defn}\label{def:tau}
		For a circle packing $\mathcal{P}=\{C_v\}_{v\in V}$ on the plane,
		define the center map
		\[
		c:V\rightarrow\CC,
		\]
		where
		\[
		c_v:=c(v)
		\]
		is the center of $C_v$ for $v\in V$.
		
		Also, if $|c_v|>r_v$ we define
		\[
		\tau(v):=\frac{r_v}{|c_v|-r_v},
		\]
		which is the quantity associated with Lemma \ref{lem:extended_ring}.  
	\end{defn}

	Let $G=(V,E)$ be a connected, locally finite, simple graph.
	In our applications, $G$ will be the $1$-skeleton of a
	triangulation. For $W\subseteq V$, let $G[W]$ denote the
	subgraph induced by $W$, with vertex set $W$ and edge set
	\[
	E(G[W])=\big\{\{u,v\}\in E : u,v\in W\big\}.
	\]
	
	We now introduce the notion of Gromov--Hausdorff convergence
	for doubly rooted graphs, which is a straightforward
	generalization of the corresponding notion for rooted graphs;
	see \cite[Definition~2.1]{MR4829680}.
	
	A doubly rooted graph is a triple $(G,o,o')$, where
	$o\sim o'$. The two roots are ordered.
	An isomorphism between two doubly rooted graphs is a graph
	isomorphism preserving the order of the roots.
		
	\begin{defn}[Doubly pointed Gromov--Hausdorff convergence]
		\label{def:graph-pgh}
		A sequence of doubly rooted graphs $(G_i,o_i,o_i')$ converges to
		$(G_\infty,o_\infty,o_\infty')$ in the doubly pointed
		Gromov--Hausdorff sense, written
		\[
		(G_i,o_i,o_i')
		\xrightarrow{\mathrm{pGH}}
		(G_\infty,o_\infty,o_\infty'),
		\]
		if, for every integer $n\geq1$, there exists $N(n)\in\mathbb N$
		such that, for every $i\geq N(n)$, there is a doubly rooted graph
		isomorphism
		\[
		\phi_{i,n}:
		\bigl(G_i[B_n^{G_i}(o_i)],o_i,o_i'\bigr)
		\longrightarrow
		\bigl(G_\infty[B_n^{G_\infty}(o_\infty)],
		o_\infty,o_\infty'\bigr).
		\]
	\end{defn}
	
	\begin{defn}[Bounded geometry]
		A graph $G$ has bounded geometry if its vertex degrees are
		uniformly bounded.
		For $B<\infty$, let $\mathcal{BG}^{r}(B)$ denote the set of
		doubly rooted isomorphism classes of connected simple graphs
		$(G,o,o')$ satisfying $o\sim o'$ and
		\[
		\deg_G(v)\leq B
		\qquad\text{for every }v\in V(G).
		\]
	\end{defn}
	
	\begin{rem}
		For a sequence of oriented triangulations without boundary
		with a uniform degree bound, the diagonal argument
		can be applied while retaining face incidences and cyclic
		orders of neighbors. The resulting subsequential local
		limit is again a locally finite oriented triangulation without boundary.
	\end{rem}
	
	\begin{lem}[Sequential compactness; compare {\cite[Lemma~2.3]{MR4829680}}]
		\label{lem:graph-pgh-compactness}
		For every fixed $B<\infty$, the space $\mathcal{BG}^{r}(B)$
		is sequentially compact in the doubly pointed
		Gromov--Hausdorff topology.
		More precisely, every sequence
		\[
		(G_i,o_i,o_i')\in\mathcal{BG}^{r}(B)
		\]
		admits a subsequence $(G_{i_j},o_{i_j},o_{i_j}')$ and a limit
		$(G_\infty,o_\infty,o_\infty')\in\mathcal{BG}^{r}(B)$ such that
		\[
		(G_{i_j},o_{i_j},o_{i_j}')
		\xrightarrow{\mathrm{pGH}}
		(G_\infty,o_\infty,o_\infty').
		\]
	\end{lem}
	
	\begin{proof}
		For each fixed $n\geq1$, the degree bound implies that there
		are only finitely many isomorphism classes of radius-$n$
		balls with an ordered pair of adjacent roots.
		The usual diagonal argument therefore gives the conclusion,
		just as in the singly rooted case.
	\end{proof}

Now we prove Lemma \ref{lem:extended_ring}.

\begin{proof}[Proof of Lemma \ref{lem:extended_ring}]
	Suppose, for a contradiction, that there exist an integer $N$
	and a constant $\tau_0>0$ for which no uniform bound holds.
	Then there is a sequence of locally univalent circle packings
	\[
	\mathcal{P}_i=\{C_v^{(i)}\}_{v\in V_i}
	\]
	of locally finite triangulations
	$\mathcal{T}_i=(V_i,E_i,F_i)$ of $\mathbb C$, satisfying
	\[
	\sup_{v\in V_i}\deg_{\mathcal{T}_i}(v)\le N,\qquad
	0\notin\operatorname{Carr}_{\CC}(\mathcal{P}_i),\qquad
	\operatorname{SV}(\mathcal{P}_i)=\{0\},\qquad
	\tau(\mathcal{P}_i)\le\tau_0,
	\]
	and adjacent vertices $v_i,w_i\in V_i$ such that
	\begin{equation}\label{equ:blow_up}
		\frac{r_{v_i}^{(i)}}{r_{w_i}^{(i)}}\longrightarrow+\infty.
	\end{equation}
	Here $c_v^{(i)}$ and $r_v^{(i)}$ are the center and radius
	of $C_v^{(i)}$, respectively.
	
	Let $G_i$ be the $1$-skeleton of $\mathcal{T}_i$.
	By Lemma~\ref{lem:graph-pgh-compactness} and a further
	diagonal extraction, the doubly rooted graphs
	$(G_i,v_i,w_i)$ converge locally to
	$(G_\infty,v_\infty,w_\infty)$.
	We retain the face incidences and the cyclic orders of
	neighbors in this extraction.
	This is possible because, under the degree bound, there
	are only finitely many such decorated balls of each
	fixed radius.
	Thus the limit carries a triangulation
	$\mathcal{T}_\infty=(V_\infty,E_\infty,F_\infty)$, and,
	after relabeling, there are local isomorphisms
	\[
	\phi_i:
	\bigl(G_i[B_i^{G_i}(v_i)],v_i,w_i\bigr)
	\longrightarrow
	\bigl(G_\infty[B_i^{G_\infty}(v_\infty)],
	v_\infty,w_\infty\bigr)
	\]
	preserving faces and cyclic orders wherever the relevant
	stars are contained in these balls.
	
	Normalize the packings by
	\[
	\Lambda_i(z):=\frac{z}{r_{v_i}^{(i)}},
	\qquad
	\widehat{\mathcal{P}}_i:=\Lambda_i(\mathcal{P}_i).
	\]
	Write $c^i$ and $r^i$ for the center and radius maps of
	$\widehat{\mathcal{P}}_i$, and set
	\[
	M_i(z):=z-c^i(v_i).
	\]
	The circle corresponding to $v_i$ in
	$M_i(\widehat{\mathcal{P}}_i)$ is the unit circle
	centered at the origin.
	
	The unique finite singular value of this normalized
	packing is
	\[
	a_i=M_i(\Lambda_i(0))=-c^i(v_i).
	\]
	Since $\tau(\mathcal{P}_i)\le\tau_0$, we have
	\[
	|a_i|
	=\frac{|c_{v_i}^{(i)}|}{r_{v_i}^{(i)}}
	\ge 1+\frac1{\tau_0}.
	\]
	Consequently, the disk
	\[
	B:=B_{\mathbb C}\left(0,1+\frac1{2\tau_0}\right)
	\]
	contains no finite singular value.
	
	Let $S_i$ be the carrier surface of
	$M_i(\widehat{\mathcal{P}}_i)$, with projection
	$\pi_i:S_i\to\mathbb C$.
	We claim that
	\[
	\pi_i:\pi_i^{-1}(B)\longrightarrow B
	\]
	is a covering map.
	Indeed, for each $a\in B$, choose an open disk
	$W\subset B$ about $a$ such that every component of
	$\pi_i^{-1}(W)$ is relatively compact.
	Choose a smaller disk $D$ about $a$ with
	$\overline D\subset W$.
	Every component $Q$ of $\pi_i^{-1}(D)$ is contained
	in a relatively compact component of $\pi_i^{-1}(W)$.
	It follows that $\pi_i|_Q:Q\to D$ is proper.
	Since $\pi_i$ is a local homeomorphism, this restriction
	is a covering onto $D$, and hence a homeomorphism.
	This proves the claim.
	
	As $B$ is simply connected, every component of
	$\pi_i^{-1}(B)$ maps homeomorphically onto $B$.
	Let $U_i$ be the component containing the vertex $v_i$.
	In particular, the carrier is embedded on $U_i$.
	
	For vertices in the domain of $\phi_i^{-1}$, define
	\[
	r_v^i:=r^i(\phi_i^{-1}(v)),
	\qquad
	z_i(v):=M_i\bigl(c^i(\phi_i^{-1}(v))\bigr).
	\]
	After a further diagonal extraction,
	\[
	r_v^i\longrightarrow r_v^\infty\in[0,+\infty]
	\qquad\text{for every fixed }v\in V_\infty,
	\]
	where the limit can be $+\infty$.
	By the normalization and \eqref{equ:blow_up},
	\[
	r_{v_\infty}^\infty=1,
	\qquad
	r_{w_\infty}^\infty=0.
	\]
	Set
	\[
	W:=\{v\in V_\infty:r_v^\infty=0\},
	\]
	and let $W_0$ be the connected component of the induced
	graph on $W$ containing $w_\infty$.
	
	Every vertex $v\in V_\infty\setminus W$ has at least two
	neighbors outside $W$.
	To see this, suppose otherwise.
	For every face $\{v,a,b\}$ incident to $v$, at least one
	of $a,b$ belongs to $W$.
	If, for example, $b\in W$, then
	\[
	\sin^2\frac{\theta_v^{ab}(i)}2
	=
	\frac{r_a^i r_b^i}
	{(r_v^i+r_a^i)(r_v^i+r_b^i)}
	\le
	\frac{r_b^i}{r_v^i+r_b^i}
	\longrightarrow0.
	\]
	Thus every angle at $v$ tends to zero.
	This contradicts the angle sum $2\pi$ at $v$, since its
	degree is finite.
	
	Define
	\[
	V_+:=
	\{v\in V_\infty\setminus W:
	v\sim w\text{ for some }w\in W_0\},
	\]
	and let $W_+$ be the connected component of the induced
	graph on $V_+$ containing $v_\infty$.
	Let $H$ be the graph with vertex set $W_+$ whose edges
	are those $\{v,a\}$ for which
	\[
	\{v,a,w\}\in F_\infty
	\quad\text{for some }w\in W_0.
	\]
	We claim that every vertex of $H$ has degree at least two.
	
	Fix $v\in W_+$.
	Its cyclic list of neighbors contains a vertex in $W_0$
	and at least two vertices outside $W_0$.
	Hence there are two distinct vertices $a,b\notin W_0$
	which are consecutive, in this cyclic list, to vertices
	of $W_0$.
	There exist $y_a,y_b\in W_0$ such that
	\[
	\{v,a,y_a\},\ \{v,b,y_b\}\in F_\infty.
	\]
	Neither $a$ nor $b$ belongs to $W$, since otherwise its
	adjacency to $W_0$ would place it in $W_0$.
	Therefore $a,b\in V_+$, and their adjacency to $v$
	implies $a,b\in W_+$.
	Thus $\deg_H(v)\ge2$.
	
	Starting at $x_0=v_\infty$, choose successive vertices
	$x_1,x_2,x_3$ in $H$ with
	\[
	x_2\ne x_0,\qquad x_3\ne x_1.
	\]
	Then the edges
	\[
	e_j=\{x_{j-1},x_j\},\qquad j=1,2,3,
	\]
	are pairwise distinct; we allow $x_3=x_0$.
	For each $j$, choose $y_j\in W_0$ such that
	\[
	T_j:=\{x_{j-1},x_j,y_j\}\in F_\infty.
	\]
	These three faces are distinct, since each has exactly
	one vertex in $W_0$, with opposite edge $e_j$.
	
	Choose a finite connected subgraph $K\subset W_0$
	containing $w_\infty,y_1,y_2,y_3$.
	For sufficiently large $i$, its realization is
	\[
	K_i=
	\{z_i(v):v\in V(K)\}
	\cup
	\bigcup_{\{u,v\}\in E(K)}[z_i(u),z_i(v)].
	\]
	Its total edge length satisfies
	\[
	\ell_i(K):=
	\sum_{\{u,v\}\in E(K)}
	|z_i(u)-z_i(v)|
	=
	\sum_{\{u,v\}\in E(K)}(r_u^i+r_v^i)
	\longrightarrow0.
	\]
	Moreover,
	\[
	|z_i(w_\infty)|=1+r_{w_\infty}^i\longrightarrow1.
	\]
	Passing to a subsequence, we may assume that
	$z_i(w_\infty)\to x$ for some $|x|=1$.
	Connectivity of $K$ gives
	\[
	\max_{v\in V(K)}|z_i(v)-x|\longrightarrow0.
	\]
	
	The root edge joining $v_i$ to $w_i$ lies over $B$
	for all sufficiently large $i$.
	Since $\ell_i(K)\to0$ and $x\in B$, every point of the
	corresponding realization of $K$ in $S_i$ can be joined
	to this root edge by a path whose image stays in $B$.
	Thus this realization lies in $U_i$.
	
	Fix $v^*\in V(K)$.
	Since $x_0,x_1,x_2,x_3\notin W$, there exists $a>0$
	such that
	\[
	r_{x_\ell}^i\ge a,
	\qquad \ell=0,1,2,3,
	\]
	for all sufficiently large $i$.
	Choose
	\[
	0<\rho<
	\min\left\{\frac a4,\frac1{8\tau_0}\right\},
	\qquad
	B_i:=B_{\CC}(z_i(v^*),\rho).
	\]
	For sufficiently large $i$, we have
	\[
	\overline{B_i}\subset B,\qquad
	z_i(y_j)\in B_i,\qquad
	z_i(x_{j-1}),z_i(x_j)\notin\overline{B_i}.
	\]
	The last assertion follows from tangency and
	$z_i(y_j)-z_i(v^*)\to0$.
	
	Let $p_j^i,q_j^i$ be the unique intersection points of
	$\partial B_i$ with the segments
	\[
	[z_i(y_j),z_i(x_{j-1})],
	\qquad
	[z_i(y_j),z_i(x_j)],
	\]
	respectively.
	Define the chord
	\[
	X_j^i:=[p_j^i,q_j^i].
	\]
	By convexity, this chord lies in the realization of $T_j$,
	and its relative interior lies in the interior of that
	triangle.
	Furthermore, the truncated triangle
	\[
	\Delta_j^i:=
	\operatorname{conv}\{z_i(y_j),p_j^i,q_j^i\}
	\]
	lies in $\overline{B_i}$.
	Its lift in the carrier is connected to $y_j$ inside
	$\pi_i^{-1}(B)$, and hence lies in $U_i$.
	All three truncated triangles and $K_i$ therefore belong
	to the same embedded sheet.
	
	The chord $X_j^i$ cannot be a diameter of $B_i$:
	otherwise its relative interior would contain
	$z_i(v^*)$, a vertex of the triangulation, while lying
	in the interior of a face.
Let $L_j^i$ be the shorter arc of $\partial B_i$
joining $p_j^i$ and $q_j^i$.
	
	We claim that the three arcs $L_j^i$ have pairwise
	disjoint relative interiors, as shown in Figure \ref{fig:Bi-disjoint}.
	Suppose otherwise that $L_j^i$ and $L_k^i$ overlap
	in their relative interiors, with $j\ne k$.
	The chords cannot cross, since their relative interiors
	lie in distinct embedded faces.
	Consequently, after interchanging $j$ and $k$
	if necessary,
	\[
	L_j^i\subset L_k^i,
	\]
	as shown in Figure \ref{fig:Bi-nested}.
	The relative interior of $X_j^i$ then lies in the open
	half-plane bounded by the line through $X_k^i$
	which does not contain $z_i(v^*)$.
	
	The vertex $z_i(y_j)$ must lie in this same half-plane.
	It cannot lie on $X_k^i$, since the relative interior
	of that chord lies in the interior of $T_k$.
	If it lay in the other half-plane, a segment from
	$z_i(y_j)$ to an interior point of $X_j^i$
	would cross the relative interior of $X_k^i$ inside $B_i$.
	The crossing point would belong to the interiors of
	both $T_j$ and $T_k$, contradicting embeddedness.
	\begin{figure}[htbp]
		\centering
		\begin{subfigure}[c]{0.58\linewidth}
			\centering
			\includegraphics[width=\linewidth]{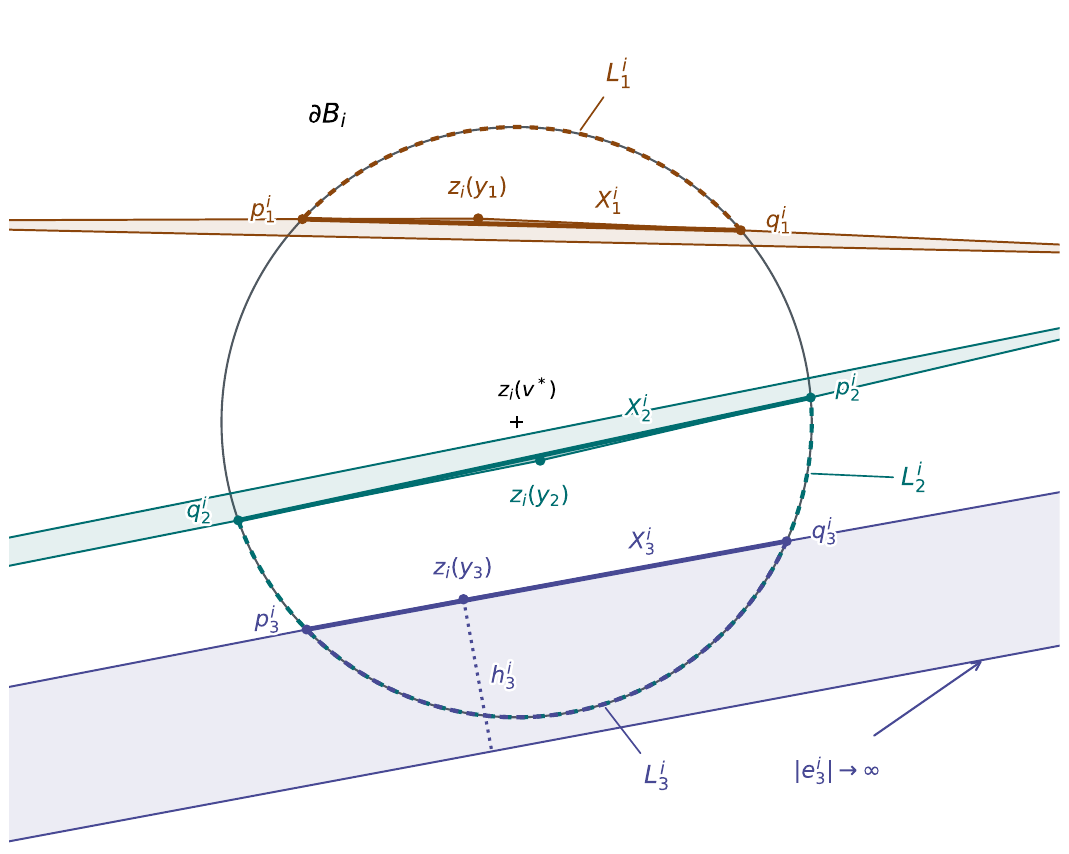}
			\caption{The arc $L_3^i$ is contained in $L_2^i$.}
			\label{fig:Bi-nested}
		\end{subfigure}\hfill
		\begin{subfigure}[c]{0.42\linewidth}
			\centering
			\includegraphics[width=\linewidth]{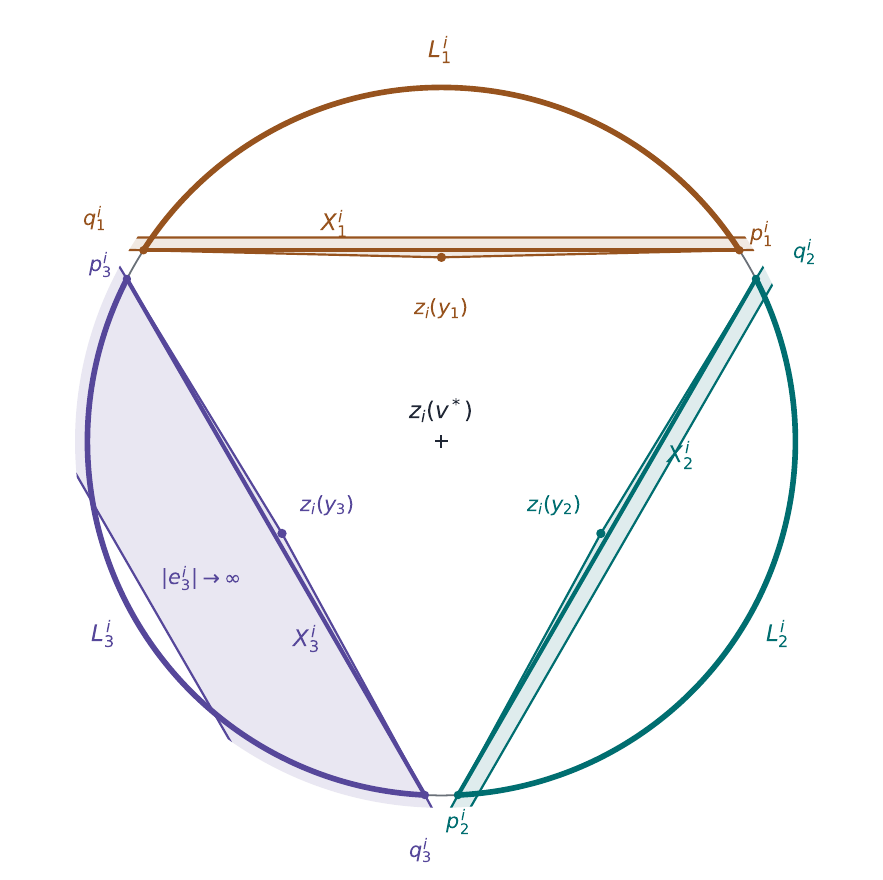}
			\caption{Arcs with pairwise disjoint interiors.}
			\label{fig:Bi-disjoint}
		\end{subfigure}
		\caption{The chords $X_j^i$ in $B_i$ and the corresponding
			shorter arcs $L_j^i$ on $\partial B_i$.}
		\label{fig:Bi-chords}
	\end{figure}
	Thus $X_k^i$ separates $z_i(y_j)$ from $z_i(v^*)$
	within $B_i$.
	An edge path in $K_i$ joining these two vertices cannot
	cross the relative interior of $X_k^i$, which lies
	in the interior of a face.
	It must therefore meet $\partial B_i$ and have length
	at least $\rho$.
	On the other hand, a simple such path has length at most
	$\ell_i(K)\to0$, a contradiction.
	This proves the claim, and hence
	\[
	\sum_{j=1}^{3}\operatorname{length}(L_j^i)\le2\pi\rho.
	\]
	
	It remains to compute the limiting arc lengths.
	Let $\alpha_j^i$ be the angle of $T_j$ at $z_i(y_j)$.
	The tangency relations give
	\[
	\sin^2\frac{\alpha_j^i}{2}
	=
	\frac{r_{x_{j-1}}^i r_{x_j}^i}
	{(r_{y_j}^i+r_{x_{j-1}}^i)
		(r_{y_j}^i+r_{x_j}^i)}
	\longrightarrow1.
	\]
	Therefore $\alpha_j^i\to\pi$.
	Since $z_i(y_j)-z_i(v^*)\to0$, the two unit vectors
	\[
	\frac{p_j^i-z_i(v^*)}{\rho},
	\qquad
	\frac{q_j^i-z_i(v^*)}{\rho}
	\]
	differ by quantities tending to zero from the
	corresponding unit vectors based at $z_i(y_j)$.
	Their mutual angle consequently tends to $\pi$.
	Thus
	\[
	\operatorname{length}(L_j^i)\longrightarrow\pi\rho,
	\qquad j=1,2,3.
	\]
	Passing to the limit in the arc-length inequality yields
	\[
	3\pi\rho\le2\pi\rho,
	\]
	which is impossible.
\end{proof}

\textbf{Acknowledgements.} 
B. Hua is supported by NSFC, no.12371056, and by Shanghai Science and Technology Program [Project No. 22JC1400100]. Puchun Zhou was supported in part by the European Union (ERC, GeoFEM, 101164551). Views and opinions expressed are however those of the authors only and do not necessarily reflect those of the European Union or the European Research Council. Neither the European Union nor the granting authority can be held responsible for them. Puchun Zhou would like to thank Xin Nie for helpful discussions on circle packings whose visual angles, viewed from the origin, are uniformly bounded away from $\pi$.

\bibliographystyle{plain}
\bibliography{reference}

@article {MR4829680,
    AUTHOR = {Hua, Bobo and M\"unch, Florentin},
     TITLE = {Graphs with nonnegative curvature outside a finite subset,
              harmonic functions, and number of ends},
   JOURNAL = {J. Lond. Math. Soc. (2)},
  FJOURNAL = {Journal of the London Mathematical Society. Second Series},
    VOLUME = {110},
      YEAR = {2024},
    NUMBER = {6},
     PAGES = {Paper No. e70034, 19},
      ISSN = {0024-6107,1469-7750},
   MRCLASS = {53C21 (31C20 53A70 53C23)},
  MRNUMBER = {4829680},
       DOI = {10.1112/jlms.70034},
       URL = {https://doi.org/10.1112/jlms.70034},
}

@misc{GeYuZhou2025,
  author       = {Ge, Huabin and Yu, Hao and Zhou, Puchun},
  title        = {Infinite ideal polyhedra in hyperbolic 3-space:
                  existence and rigidity},
  year         = {2025},
  howpublished = {arXiv:2506.19528},
  doi          = {10.48550/arXiv.2506.19528},
  url          = {https://arxiv.org/abs/2506.19528}
}

@misc{Nie2026,
  author = {Nie, Xin},
  year   = {2026},
  note   = {Private communication on an analytic proof
            of the rigidity of Doyle spirals for locally univalent
            circle packings},
}

@misc{Eremenko2021Inverse,
  author       = {Eremenko, Alexandre},
  title        = {Singularities of inverse functions},
  year         = {2021},
  howpublished = {arXiv:2110.06134},
  doi          = {10.48550/arXiv.2110.06134},
  url          = {https://arxiv.org/abs/2110.06134}
}

@article {MR1419007,
    AUTHOR = {Benjamini, I. and Schramm, O.},
     TITLE = {Harmonic functions on planar and almost planar graphs and
              manifolds, via circle packings},
   JOURNAL = {Invent. Math.},
  FJOURNAL = {Inventiones Mathematicae},
    VOLUME = {126},
      YEAR = {1996},
    NUMBER = {3},
     PAGES = {565--587},
      ISSN = {0020-9910,1432-1297},
   MRCLASS = {31C20 (30F20 52C15)},
  MRNUMBER = {1419007},
MRREVIEWER = {Sam\ Northshield},
       DOI = {10.1007/s002220050109},
       URL = {https://doi.org/10.1007/s002220050109},
}

@phdthesis{Callahan1993,
  author = {Kevin E. Callahan},
  title = {The Application of Ahlfors' Value Distribution
           Theory of Covering Surfaces to Circle Packing},
  school = {University of California, San Diego},
  year = {1993}
}

@book {MR2963949,
    AUTHOR = {Brilleslyper, M. A. and Dorff, M. J. and McDougall,
              J. M. and Rolf, J. S. and Schaubroeck, L. E. and
              Stankewitz, Ri. L. and Stephenson, K.},
     TITLE = {Explorations in complex analysis},
    SERIES = {Classroom Resource Materials Series},
 PUBLISHER = {Mathematical Association of America, Washington, DC},
      YEAR = {2012},
     PAGES = {xviii+373},
      ISBN = {978-0-88385-778-6; 978-1-61444-108-3},
   MRCLASS = {30-01},
  MRNUMBER = {2963949},
}

@article{doyl_study2,
    author = {Emch, A.},
    title = {Mathematics and engineering in nature},
    journal = {Popular Science Monthly},
    year = {1911}
}

@article{doyl_study1,
    author = {Emch, A.},
    title = {Sur quelques exemples mathématiques dans les sciences naturelles},
    journal = {L'Enseignement mathématique},
    year = {1910}
}

@article {MR1781842,
    AUTHOR = {Jonasson, J. and Schramm, O.},
     TITLE = {On the cover time of planar graphs},
   JOURNAL = {Electron. Comm. Probab.},
  FJOURNAL = {Electronic Communications in Probability},
    VOLUME = {5},
      YEAR = {2000},
     PAGES = {85--90},
      ISSN = {1083-589X},
   MRCLASS = {60J10 (05C85 52C15 60G50)},
  MRNUMBER = {1781842},
MRREVIEWER = {Robert\ P.\ Dobrow},
       DOI = {10.1214/ECP.v5-1022},
       URL = {https://doi.org/10.1214/ECP.v5-1022},
}

@article {MR1438463,
    AUTHOR = {Miller, G. L. and Teng, S.-H and Thurston, W.P. and
              Vavasis, S. A.},
     TITLE = {Separators for sphere-packings and nearest neighbor graphs},
   JOURNAL = {J. ACM},
  FJOURNAL = {Journal of the ACM},
    VOLUME = {44},
      YEAR = {1997},
    NUMBER = {1},
     PAGES = {1--29},
      ISSN = {0004-5411,1557-735X},
   MRCLASS = {52C17 (05C10 52B55 68U05)},
  MRNUMBER = {1438463},
MRREVIEWER = {Rade\ \v{Z}ivaljevi\'{c}},
       DOI = {10.1145/256292.256294},
       URL = {https://doi.org/10.1145/256292.256294},
}

@book {MR3970273,
    AUTHOR = {Nachmias, A.},
     TITLE = {Planar maps, random walks and circle packing},
    SERIES = {Lecture Notes in Mathematics},
    VOLUME = {2243},
      NOTE = {\'{E}cole d'\'{e}t\'{e} de probabilit\'{e}s de Saint-Flour
              XLVIII---2018,
              \'{E}cole d'\'{E}t\'{e} de Probabilit\'{e}s de Saint-Flour.
              [Saint-Flour Probability Summer School]},
 PUBLISHER = {Springer, Cham},
      YEAR = {2020},
     PAGES = {xii+118},
      ISBN = {978-3-030-27967-7; 978-3-030-27968-4},
   MRCLASS = {60C05 (05C10 05C81 52C26 60D05 82B41)},
  MRNUMBER = {3970273},
MRREVIEWER = {Nicolas\ Curien},
       DOI = {10.1007/978-3-030-27968-4},
       URL = {https://doi.org/10.1007/978-3-030-27968-4},
}

@article {convergence_discrete_conformal,
    AUTHOR = {Luo, F. and Sun, J. and Wu, T.},
     TITLE = {Discrete conformal geometry of polyhedral surfaces and its
              convergence},
   JOURNAL = {Geom. Topol.},
  FJOURNAL = {Geometry \& Topology},
    VOLUME = {26},
      YEAR = {2022},
    NUMBER = {3},
     PAGES = {937--987},
      ISSN = {1465-3060,1364-0380},
   MRCLASS = {52C26 (58E30)},
  MRNUMBER = {4466644},
MRREVIEWER = {Ulrike\ B\"{u}cking},
       DOI = {10.2140/gt.2022.26.937},
       URL = {https://doi.org/10.2140/gt.2022.26.937},
}

@article {Beardon_stephenson,
    AUTHOR = {Beardon, A. F. and Stephenson, K.},
     TITLE = {The uniformization theorem for circle packings},
   JOURNAL = {Indiana Univ. Math. J.},
  FJOURNAL = {Indiana University Mathematics Journal},
    VOLUME = {39},
      YEAR = {1990},
    NUMBER = {4},
     PAGES = {1383--1425},
      ISSN = {0022-2518,1943-5258},
   MRCLASS = {52C15 (20H10 30F10 30F99 57M50)},
  MRNUMBER = {1087197},
MRREVIEWER = {Darryl\ McCullough},
       DOI = {10.1512/iumj.1990.39.39062},
       URL = {https://doi.org/10.1512/iumj.1990.39.39062},
}

@article{jean1983introductory,
  title={Introductory review: Mathematical modeling in phyllotaxis: The state of the art},
  author={Jean, R. V.},
  journal={Mathematical Biosciences},
  volume={64},
  number={1},
  pages={1--27},
  year={1983},
  publisher={Elsevier}
}

@article{wu2015rigidity,
  title={Rigidity of infinite hexagonal triangulation of the plane},
  author={Wu, T. and Gu, X. and Sun, J.},
  journal={Transactions of the American Mathematical Society},
  volume={367},
  number={9},
  pages={6539--6555},
  year={2015}
}

@article {Schramm_oded,
    AUTHOR = {Schramm, O.},
     TITLE = {Rigidity of infinite (circle) packings},
   JOURNAL = {J. Amer. Math. Soc.},
  FJOURNAL = {Journal of the American Mathematical Society},
    VOLUME = {4},
      YEAR = {1991},
    NUMBER = {1},
     PAGES = {127--149},
      ISSN = {0894-0347,1088-6834},
   MRCLASS = {52C15 (30C65)},
  MRNUMBER = {1076089},
MRREVIEWER = {B\'{e}la\ Uhrin},
       DOI = {10.2307/2939257},
       URL = {https://doi.org/10.2307/2939257},
}

@article {Rodin_Sullivan,
    AUTHOR = {Rodin, B. and Sullivan, D.},
     TITLE = {The convergence of circle packings to the {R}iemann mapping},
   JOURNAL = {J. Differential Geom.},
  FJOURNAL = {Journal of Differential Geometry},
    VOLUME = {26},
      YEAR = {1987},
    NUMBER = {2},
     PAGES = {349--360},
      ISSN = {0022-040X,1945-743X},
   MRCLASS = {30C35 (05B40 11H31 52A45 57N05)},
  MRNUMBER = {906396},
       URL = {http://projecteuclid.org/euclid.jdg/1214441375},
}

@article {MR1106755,
    AUTHOR = {Colin de Verdi\`ere, Y.},
     TITLE = {Un principe variationnel pour les empilements de cercles},
   JOURNAL = {Invent. Math.},
  FJOURNAL = {Inventiones Mathematicae},
    VOLUME = {104},
      YEAR = {1991},
    NUMBER = {3},
     PAGES = {655--669},
      ISSN = {0020-9910},
   MRCLASS = {57M50 (30F99 51M16 52C15)},
  MRNUMBER = {1106755},
MRREVIEWER = {N. V. Ivanov},
       DOI = {10.1007/BF01245096},
       URL = {https://doi.org/10.1007/BF01245096},
}

@article{thurston1976geometry,
  title={The Geometry and Topology of Three-Manifolds},
  author={Thurston, W. P.},
  journal={Princeton lecture notes},
  year={1976},
}

@article {bobenko1,
    AUTHOR = {Bobenko, A. I. and Hoffmann, T.},
     TITLE = {Conformally symmetric circle packings: a generalization of
              {D}oyle's spirals},
   JOURNAL = {Experiment. Math.},
  FJOURNAL = {Experimental Mathematics},
    VOLUME = {10},
      YEAR = {2001},
    NUMBER = {1},
     PAGES = {141--150},
      ISSN = {1058-6458,1944-950X},
   MRCLASS = {52C26},
  MRNUMBER = {1822860},
       URL = {http://projecteuclid.org/euclid.em/999188429},
}

@article {MR1395721,
    AUTHOR = {He, Z. and Schramm, O.},
     TITLE = {On the convergence of circle packings to the {R}iemann map},
   JOURNAL = {Invent. Math.},
  FJOURNAL = {Inventiones Mathematicae},
    VOLUME = {125},
      YEAR = {1996},
    NUMBER = {2},
     PAGES = {285--305},
      ISSN = {0020-9910,1432-1297},
   MRCLASS = {30C35 (52C15)},
  MRNUMBER = {1395721},
MRREVIEWER = {Kenneth\ Stephenson},
       DOI = {10.1007/s002220050076},
       URL = {https://doi.org/10.1007/s002220050076},
}

@article {MR1638772,
    AUTHOR = {He, Z. and Schramm, O.},
     TITLE = {The {$C^\infty$}-convergence of hexagonal disk packings to the
              {R}iemann map},
   JOURNAL = {Acta Math.},
  FJOURNAL = {Acta Mathematica},
    VOLUME = {180},
      YEAR = {1998},
    NUMBER = {2},
     PAGES = {219--245},
      ISSN = {0001-5962,1871-2509},
   MRCLASS = {52C15 (30C20 30E10 39A12)},
  MRNUMBER = {1638772},
MRREVIEWER = {Fr\'{e}d\'{e}ric\ Math\'{e}us},
       DOI = {10.1007/BF02392900},
       URL = {https://doi.org/10.1007/BF02392900},
}

@article {MR3457760,
    AUTHOR = {Krieg, D. and Wegert, E.},
     TITLE = {Rigidity of circle packings with crosscuts},
   JOURNAL = {Beitr. Algebra Geom.},
  FJOURNAL = {Beitr\"{a}ge zur Algebra und Geometrie. Contributions to
              Algebra and Geometry},
    VOLUME = {57},
      YEAR = {2016},
    NUMBER = {1},
     PAGES = {1--36},
      ISSN = {0138-4821,2191-0383},
   MRCLASS = {52C26 (30C80 30D40)},
  MRNUMBER = {3457760},
MRREVIEWER = {Ulrike\ B\"{u}cking},
       DOI = {10.1007/s13366-015-0245-7},
       URL = {https://doi.org/10.1007/s13366-015-0245-7},
}

@article {MR1489142,
    AUTHOR = {Dubejko, T.},
     TITLE = {Infinite branched circle packings and discrete complex
              polynomials},
   JOURNAL = {J. London Math. Soc. (2)},
  FJOURNAL = {Journal of the London Mathematical Society. Second Series},
    VOLUME = {56},
      YEAR = {1997},
    NUMBER = {2},
     PAGES = {347--368},
      ISSN = {0024-6107,1469-7750},
   MRCLASS = {05B40 (30C10 30C62 52C15)},
  MRNUMBER = {1489142},
MRREVIEWER = {Fr\'{e}d\'{e}ric\ Math\'{e}us},
       DOI = {10.1112/S0024610797005280},
       URL = {https://doi.org/10.1112/S0024610797005280},
}

@article {Bobenko2,
    AUTHOR = {Bobenko, A. I. and Hoffmann, T.},
     TITLE = {Hexagonal circle patterns and integrable systems: patterns
              with constant angles},
   JOURNAL = {Duke Math. J.},
  FJOURNAL = {Duke Mathematical Journal},
    VOLUME = {116},
      YEAR = {2003},
    NUMBER = {3},
     PAGES = {525--566},
      ISSN = {0012-7094,1547-7398},
   MRCLASS = {52C26 (34M55 37K10 37K20 37K60 39A12)},
  MRNUMBER = {1958097},
MRREVIEWER = {Yuri\ B.\ Suris},
       DOI = {10.1215/S0012-7094-03-11635-X},
       URL = {https://doi.org/10.1215/S0012-7094-03-11635-X},
}

@article {circlepattern_square,
    AUTHOR = {Schramm, O.},
     TITLE = {Circle patterns with the combinatorics of the square grid},
   JOURNAL = {Duke Math. J.},
  FJOURNAL = {Duke Mathematical Journal},
    VOLUME = {86},
      YEAR = {1997},
    NUMBER = {2},
     PAGES = {347--389},
      ISSN = {0012-7094,1547-7398},
   MRCLASS = {30G25 (05B40 30D99 52C15)},
  MRNUMBER = {1430437},
MRREVIEWER = {Philip\ L.\ Bowers},
       DOI = {10.1215/S0012-7094-97-08611-7},
       URL = {https://doi.org/10.1215/S0012-7094-97-08611-7},
}

@book{woess2000random,
  title={Random walks on infinite graphs and groups},
  author={Woess, W.},
  number={138},
  year={2000},
  publisher={Cambridge university press}
}

@book{grigor2018introduction,
  title={Introduction to analysis on graphs},
  author={Grigor’yan, A.},
  volume={71},
  year={2018},
  publisher={American Mathematical Soc.}
}

@article{koebe1936origin,
  title={Kontaktprobleme der konformen Abbildung},
  author={Koebe, P.},
  year={1936},
Journal={Ber. S\"achs. Akad. Wiss. Leipzig Math. -Phys. Kl. },
volume={88},
year = {1936},
pages = {141-164},
  publisher={Hirzel Stuttgart}
}

@article {Rodin93,
    AUTHOR = {Callahan, K. and Rodin, B.},
     TITLE = {Circle packing immersions form regularly exhaustible surfaces},
   JOURNAL = {Complex Variables Theory Appl.},
  FJOURNAL = {Complex Variables. Theory and Application. An International
              Journal},
    VOLUME = {21},
      YEAR = {1993},
    NUMBER = {3-4},
     PAGES = {171--177},
      ISSN = {0278-1077,1563-5066},
   MRCLASS = {30D35 (52C15)},
  MRNUMBER = {1276573},
MRREVIEWER = {Oded\ Schramm},
       DOI = {10.1080/17476939308814626},
       URL = {https://doi.org/10.1080/17476939308814626},
}

@book{gu2008computational,
  title={Computational conformal geometry},
  author={Gu, X. and Yau, S.-T.},
  volume={1},
  year={2008},
  publisher={International Press Somerville, MA}
}

@article{andreev1970convex1,
  title={On convex polyhedra in Loba{\v{c}}evski{\u\i} spaces},
  author={Andreev, E. M.},
  journal={Mathematics of the USSR-Sbornik},
  volume={10},
  number={3},
  pages={413},
  year={1970},
  publisher={IOP Publishing}
}

@article{lovasz1999geometric,
  title={Geometric representations of graphs},
  author={Lov{\'a}sz, L. and Vesztergombi, K.},
  journal={Paul Erdos and his Mathematics},
  volume={2},
  year={1999}
}

@inproceedings{spielman1996disk,
  title={Disk packings and planar separators},
  author={Spielman, D. A. and Teng, S.},
  booktitle={Proceedings of the twelfth annual symposium on Computational geometry},
  pages={349--358},
  year={1996}
}

@article {Stephenson_D,
    AUTHOR = {Beardon, A. F. and Dubejko, T. and Stephenson, K.},
     TITLE = {Spiral hexagonal circle packings in the plane},
   JOURNAL = {Geom. Dedicata},
  FJOURNAL = {Geometriae Dedicata},
    VOLUME = {49},
      YEAR = {1994},
    NUMBER = {1},
     PAGES = {39--70},
      ISSN = {0046-5755,1572-9168},
   MRCLASS = {52C15 (30C99)},
  MRNUMBER = {1261573},
MRREVIEWER = {Oded\ Schramm},
       DOI = {10.1007/BF01263534},
       URL = {https://doi.org/10.1007/BF01263534},
}

@article{ge_phd,
title={Combinatorial methods and geometric equations},
author={Ge, H.},
journal={Ph.D. Thesis, Peking University, Beijing},
year={2012}
}

@book {Stephenson_intro,
    AUTHOR = {Stephenson, K.},
     TITLE = {Introduction to circle packing},
      NOTE = {The theory of discrete analytic functions},
 PUBLISHER = {Cambridge University Press, Cambridge},
      YEAR = {2005},
     PAGES = {xii+356},
      ISBN = {978-0-521-82356-2; 0-521-82356-0},
   MRCLASS = {52C26 (30G25)},
  MRNUMBER = {2131318},
MRREVIEWER = {Fr\'{e}d\'{e}ric\ Math\'{e}us},
}

@article {He_schramm,
    AUTHOR = {He, Z. and Schramm, O.},
     TITLE = {Hyperbolic and parabolic packings},
   JOURNAL = {Discrete Comput. Geom.},
  FJOURNAL = {Discrete \& Computational Geometry. An International Journal
              of Mathematics and Computer Science},
    VOLUME = {14},
      YEAR = {1995},
    NUMBER = {2},
     PAGES = {123--149},
      ISSN = {0179-5376,1432-0444},
   MRCLASS = {52C15 (30C35 30C62 60J15)},
  MRNUMBER = {1331923},
MRREVIEWER = {Tomasz\ Dubejko},
       DOI = {10.1007/BF02570699},
       URL = {https://doi.org/10.1007/BF02570699},
}

@article {HE,
    AUTHOR = {He, Z.},
     TITLE = {Rigidity of infinite disk patterns},
   JOURNAL = {Ann. of Math. (2)},
  FJOURNAL = {Annals of Mathematics. Second Series},
    VOLUME = {149},
      YEAR = {1999},
    NUMBER = {1},
     PAGES = {1--33},
      ISSN = {0003-486X,1939-8980},
   MRCLASS = {30F20 (30C25 52C25 52C26)},
  MRNUMBER = {1680531},
MRREVIEWER = {Kenneth\ Stephenson},
       DOI = {10.2307/121018},
       URL = {https://doi.org/10.2307/121018},
}

@article {MR3010812,
    AUTHOR = {Gurel-Gurevich, O. and Nachmias, A.},
     TITLE = {Recurrence of planar graph limits},
   JOURNAL = {Ann. of Math. (2)},
  FJOURNAL = {Annals of Mathematics. Second Series},
    VOLUME = {177},
      YEAR = {2013},
    NUMBER = {2},
     PAGES = {761--781},
      ISSN = {0003-486X,1939-8980},
   MRCLASS = {60F05 (05C10 05C81 60D05)},
  MRNUMBER = {3010812},
MRREVIEWER = {Michele\ Zito},
       DOI = {10.4007/annals.2013.177.2.10},
       URL = {https://doi.org/10.4007/annals.2013.177.2.10},
}

@article {MR1207210,
    AUTHOR = {He, Z. and Schramm, O.},
     TITLE = {Fixed points, {K}oebe uniformization and circle packings},
   JOURNAL = {Ann. of Math. (2)},
  FJOURNAL = {Annals of Mathematics. Second Series},
    VOLUME = {137},
      YEAR = {1993},
    NUMBER = {2},
     PAGES = {369--406},
      ISSN = {0003-486X,1939-8980},
   MRCLASS = {30C35 (30C80 30F45 52C15)},
  MRNUMBER = {1207210},
MRREVIEWER = {Kenneth\ Stephenson},
       DOI = {10.2307/2946541},
       URL = {https://doi.org/10.2307/2946541},
}

@article {MR4389487,
    AUTHOR = {Dai, S. and Ge, H. and Ma, S.},
     TITLE = {Rigidity of the hexagonal {D}elaunay triangulated plane},
   JOURNAL = {Peking Math. J.},
  FJOURNAL = {Peking Mathematical Journal},
    VOLUME = {5},
      YEAR = {2022},
    NUMBER = {1},
     PAGES = {1--20},
      ISSN = {2096-6075,2524-7182},
   MRCLASS = {52C25 (52C26)},
  MRNUMBER = {4389487},
MRREVIEWER = {Ren\ Ding},
       DOI = {10.1007/s42543-021-00036-8},
       URL = {https://doi.org/10.1007/s42543-021-00036-8},
}

\noindent Bobo Hua, bobohua@fudan.edu.cn\\[2pt]
\emph{School of Mathematical Sciences, LMNS,
Fudan University, Shanghai 200433, China; Shanghai Center for
Mathematical Sciences, Fudan University, Shanghai, 200438, China.}
\\

\noindent Puchun Zhou, puchun.zhou@maths.ox.ac.uk\\
\emph{Mathematical Institute, University of Oxford, Oxford OX2 6GG, United Kingdom.}
\end{document}